\documentclass[12pt]{article}
\usepackage[utf8]{inputenc}
\usepackage[T1]{fontenc}
\usepackage{amsmath,amssymb,amsthm}
\usepackage{geometry}
\usepackage{hyperref}
\usepackage{indentfirst}
\usepackage{booktabs}
\usepackage{xcolor}
\usepackage{fvextra}
\DefineVerbatimEnvironment{code}{Verbatim}{
  breaklines=true,
  breakanywhere=true,
  fontsize=\footnotesize,
  xleftmargin=1em,
}
\DefineVerbatimEnvironment{term}{Verbatim}{
  breaklines=true,
  breakanywhere=true,
  fontsize=\footnotesize,
  xleftmargin=1em,
}

\newtheorem{theorem}{Theorem}
\newtheorem{lemma}{Lemma}
\newtheorem{proposition}{Proposition}
\newtheorem{corollary}{Corollary}
\theoremstyle{remark}
\newtheorem{remark}{Remark}
\theoremstyle{definition}

\title{Irreducibility of the Cuboid Polynomial $P_{a,u}(t)$ via a Rank-Zero Elliptic Curve}

\author{Valery Asiryan\\[3pt]
\small \texttt{asiryanvalery@gmail.com}}
\date{\small November 4, 2026}

\begin{document}

\maketitle


\begin{abstract}
In this paper we study the even monic degree-8 cuboid polynomial $P_{a,u}(t)$ introduced by R.~A.~Sharipov in the first-cuboid specialization of his cuboid equations. For nonzero integers $a,u$ with $u^2\neq a^2$ we prove that $P_{a,u}(t)$ is irreducible in $\mathbb{Z}[t]$ (equivalently, in $\mathbb{Q}[t]$), thus confirming Sharipov's irreducibility conjecture in this two-parameter case. Over $K=\mathbb{Q}(\sqrt2)$ we have a factorization $P_{a,u}(t)=H_-(t)H_+(t)$ into two conjugate quartics. We show that reducibility of either $H_\pm$ would force the discriminant of a certain quadratic in $S=t^2$ to be a square in $K$, which in turn implies (via $\tau=(au/\Delta)^2$) the existence of a rational point $(y,v)\in\mathcal{C}(\mathbb{Q})$ on the genus-one quartic $\mathcal{C}:\ v^2=16y^4+136y^2+1$ with $y^2=\tau$. We give an explicit isomorphism $\overline{\mathcal{C}}\simeq E$ with the elliptic curve $E:\ Y^2=X(X-8)(X-9)$, which has conductor $48$ and Mordell--Weil rank $0$. Enumerating $E(\mathbb{Q})$ and tracing back to $\mathcal{C}(\mathbb{Q})$ rules out the only possible values $\tau\in\{0,\tfrac14\}$, and hence excludes any further factorization of $H_\pm$ in $K[t]$. A quadratic Galois descent then yields the irreducibility of $P_{a,u}(t)$ over $\mathbb{Q}$ and $\mathbb{Z}$.


\end{abstract}


\section{Introduction}\label{sec:intro}
The \emph{perfect cuboid problem} asks for a rectangular box with integer edge lengths whose face diagonals and the space diagonal are also integers.

One approach, due to R.~A.~Sharipov, associates to a potential cuboid an even monic polynomial with integer parameters. In the specialization $a=b$ (the ``first cuboid'' case) this reduces to the two-parameter degree-8 polynomial $P_{a,u}(t)$ studied below (see \cite{Sharipov2011Cuboids,Sharipov2011Note}).

In \cite{Sharipov2011Cuboids} Sharipov conjectured that in this case $P_{a,u}(t)$ is irreducible in $\mathbb{Z}[t]$ for coprime $a\neq u$ (see Conjecture~5.1 therein). In \cite{Sharipov2011Note} he proved, in particular, that $P_{a,u}(t)$ has no integer roots. The main theorem of the present paper establishes the full irreducibility statement; in particular it applies to the cuboid setting $a,u\in\mathbb{Z}_{>0}$, $\gcd(a,u)=1$, $a\neq u$.

Following this approach we work with Sharipov's polynomial
\[
P_{a,u}(t)=t^8+6(u^2-a^2)\,t^6+\bigl((u^2-a^2)^2-2a^2u^2\bigr)\,t^4-6(u^2-a^2)a^2u^2\,t^2+a^4u^4\in\mathbb{Z}[t].
\]
For later use we set
\[
\Delta:=u^2-a^2\neq0,\qquad A_0:=a^2u^2,
\]
so that the displayed formula becomes
\[
P_{a,u}(t)=t^8+6\Delta\,t^6+(\Delta^2-2A_0)\,t^4-6\Delta A_0\,t^2+A_0^2.
\]

\begin{theorem}\label{thm:main}
For any integers $a,u$ with $au\neq0$ and $u^2\neq a^2$, the polynomial $P_{a,u}(t)\in\mathbb{Z}[t]$ is irreducible.
\end{theorem}

\medskip
\noindent\textbf{Idea of the proof.}
Over the quadratic field $K=\mathbb{Q}(\sqrt2)$ the polynomial $P_{a,u}(t)$ splits as $H_-(t)\,H_+(t)$, where the quartics $H_\pm$ are conjugate over $K$.
By Lemma~\ref{lem:even}, reducibility of either $H_\pm$ in $K[t]$ is equivalent to splitting of its associated quadratic in $S=t^2$, and hence to the discriminant being a square in $K$.
This, in turn, implies (using $\tau=(au/\Delta)^2$) the existence of a rational point $(y,v)\in\mathcal{C}(\mathbb{Q})$ on one fixed genus-one quartic curve $\mathcal{C}$ such that $y^2=\tau$.
Passing from $\mathcal{C}$ to its Jacobian yields an elliptic curve
\[
E:\ Y^2=X(X-8)(X-9)=X^3-17X^2+72X,
\]
of conductor $48$ and rank $0$.
We give an explicit isomorphism $\overline{\mathcal{C}}\simeq E$ over $\mathbb{Q}$ (proposition~\ref{prop:CEiso}), hence $\overline{\mathcal{C}}(\mathbb{Q})$ is obtained from the finite torsion group $E(\mathbb{Q})$.
This yields that the only values of the parameter $\tau=(au/\Delta)^2$ coming from $\mathcal{C}(\mathbb{Q})$ are $\tau\in\{0,\tfrac14\}$, which are incompatible with rational (in particular, integer) parameters $a,u$ satisfying $au\neq0$ and $u^2\neq a^2$.

Hence both $H_\pm$ are irreducible in $K[t]$, and a standard Galois-descent argument gives the irreducibility of $P_{a,u}(t)$ over $\mathbb{Q}$ (hence over $\mathbb{Z}$).

\medskip
\noindent\textbf{Outline of the paper.}
In Section~\ref{sec:split} we record the splitting $P_{a,u}=H_-H_+$ over $K$, coprimeness of the factors, and the real-root composition criterion.
Section~\ref{sec:reduction} reduces any further factorization to rational points on a fixed genus-one quartic.
In Sections~\ref{sec:jac}--\ref{sec:rank} we establish smoothness and an explicit elliptic-curve isomorphism, compute the rank and torsion, and identify the curve in standard databases (Cremona label $\mathrm{48a3}$; see \cite{CremonaEC,LMFDB48a3}).
Finally, sections~\ref{sec:CQ}--\ref{sec:final} enumerate $\mathcal{C}(\mathbb{Q})$ and prove the main irreducibility theorem.


\section{Splitting over \texorpdfstring{$K=\mathbb{Q}(\sqrt2)$}{K=Q(sqrt2)} and a composition criterion}\label{sec:split}
Let $a,u\in\mathbb{Z}$ satisfy $au\neq0$ and $u^2\neq a^2$, and put
\[
\Delta:=u^2-a^2\neq0,\qquad A_0:=a^2u^2>0.
\]
Write
\[
P_{a,u}(t)=t^8+At^6+Bt^4+Ct^2+D,\quad
A=6\Delta,\ B=\Delta^2-2A_0,\ C=-6\Delta A_0,\ D=A_0^2.
\]
Let $K:=\mathbb{Q}(\sqrt2)$.

\begin{lemma}[Explicit $K$-split]\label{lem:split}
In $K[t]$,
\[
P_{a,u}(t)=H_-(t)\,H_+(t),\qquad
H_\pm(t):=t^4+(3\pm2\sqrt2)\Delta\,t^2-A_0.
\]
\end{lemma}

\begin{proof}
With $S=t^2$,
\begin{align*}
H_-(t)H_+(t)
&=(S^2+\bigl(3-2\sqrt2\bigr)\Delta S-A_0)\,(S^2+\bigl(3+2\sqrt2\bigr)\Delta S-A_0)\\
&\text{(using $(3-2\sqrt2)+(3+2\sqrt2)=6$ and $(3-2\sqrt2)(3+2\sqrt2)=1$)}\\
&= S^4+A S^3+B S^2+C S+D.
\end{align*}
\end{proof}

\begin{lemma}[Coprimeness]\label{lem:gcd1}
$\gcd(H_-,H_+)=1$ in $K[t]$.
\end{lemma}

\begin{proof}
A common root $t_0$ gives $S_0=t_0^2$ satisfying both
\[
S^2+(3-2\sqrt2)\Delta S-A_0=0,\qquad
S^2+(3+2\sqrt2)\Delta S-A_0=0
\]
at $S=S_0$. Subtracting yields
$4\sqrt2\,\Delta\,S_0=0\Rightarrow S_0=0\Rightarrow A_0=0$, impossible.
\end{proof}


\begin{lemma}[A real-root criterion for composition with a square]\label{lem:even}
Let $F\subset\mathbb{R}$ be a field, let $\alpha,A_0\in F$, and suppose $A_0>0$.
Then the quartic
\[
Q(t)=t^4+\alpha t^2-A_0
\]
is irreducible in $F[t]$ if and only if the quadratic
\[
q(S)=S^2+\alpha S-A_0
\]
is irreducible in $F[S]$.
\end{lemma}

\begin{proof}
If $q$ is reducible in $F[S]$, substituting $S=t^2$ into its factorization gives a nontrivial factorization of $Q$ in $F[t]$.
Conversely, suppose that $q$ is irreducible.
Its two roots are real and have product $-A_0<0$, so one root $\beta$ is negative.
The field $L=F(\beta)\subset\mathbb{R}$ has degree $2$ over $F$.
Since a square in $L$ is nonnegative, $\beta$ is not a square in $L$, and $t^2-\beta$ is irreducible over $L$.
If $\gamma^2=\beta$, then $L\subset F(\gamma)$ and
\[
[F(\gamma):F]=[F(\gamma):L]\,[L:F]=2\cdot2=4.
\]
The element $\gamma$ is a root of the monic degree-$4$ polynomial $Q$, which is therefore irreducible over $F$.
\end{proof}

\section{Reduction to a fixed genus-one quartic}\label{sec:reduction}
Put $S:=t^2$ and define
\begin{equation}\label{eq:DeltaS}
h_\pm(S):=S^2+(3\pm2\sqrt2)\Delta\,S-A_0\in K[S],\qquad
D_\pm:=\mathrm{Disc}(h_\pm)=(17\pm12\sqrt2)\Delta^2+4A_0\in K.
\end{equation}
Then $H_\pm(t)=h_\pm(t^2)$. Moreover, $h_\pm$ splits in $K[S]$ if and only if $D_\pm$ is a square in $K$.
We treat $h_-$; the case of $h_+$ follows by applying the conjugation $\sqrt2\mapsto-\sqrt2$.

Assume that $h_-$ splits in $K[S]$. Then $D_-=(r+w\sqrt2)^2$ for some $r,w\in\mathbb{Q}$.
Comparing the rational and $\sqrt2$-parts yields
\begin{equation}\label{eq:rs}
2rw=-12\Delta^2,\qquad r^2+2w^2=17\Delta^2+4A_0.
\end{equation}
Eliminating $r$ and writing $T:=w^2\in\mathbb{Q}_{\ge0}$ we obtain
\begin{equation}\label{eq:Tquad}
2T^2-(17\Delta^2+4A_0)T+36\Delta^4=0.
\end{equation}
Hence its discriminant
\begin{equation}\label{eq:bigD}
\Delta_T:=(17\Delta^2+4A_0)^2-288\Delta^4
=\Delta^4+136\Delta^2A_0+16A_0^2
\end{equation}
must be a rational square; we write $Z\in\mathbb{Q}$ for a square root, so that $Z^2=\Delta_T$.

Introduce
\[
X_0:=\Delta^2,\quad v:=\frac{Z}{X_0},\quad
\tau:=\frac{A_0}{X_0}=\Bigl(\frac{au}{\Delta}\Bigr)^2\in\mathbb{Q}_{>0}.
\]
Dividing \eqref{eq:bigD} by $X_0^2$ yields the conic
\begin{equation}\label{eq:conic}
v^2=16\tau^2+136\tau+1.
\end{equation}
Since $\tau=(au/\Delta)^2$ is a rational square, write $\tau=y^2$, and we arrive at the quartic
\begin{equation}\label{eq:quartic}
\mathcal{C}:\quad v^2=16y^4+136y^2+1.
\end{equation}
As shown in lemma~\ref{lem:genus1}, the smooth projective model $\overline{\mathcal{C}}$ is a smooth curve of genus $1$.

\section{The fixed genus-one curve and its elliptic model}\label{sec:jac}
Let $\overline{\mathcal{C}}$ be the projective closure of $\mathcal{C}$ in the weighted projective space $\mathbb{P}(1,2,1)$ with coordinates $(y:v:z)$ of weights $1,2,1$:
\[
\overline{\mathcal{C}}:\quad v^2=16y^4+136y^2z^2+z^4.
\]
Since the leading and trailing coefficients are squares, $\overline{\mathcal{C}}$ has two rational points at infinity,
\[
P_\infty^\pm:=(1:\pm4:0),
\]
corresponding to the branches $v=\pm4y^2$ at $z=0$.

\begin{lemma}[Smoothness and genus]\label{lem:genus1}
The curve $\overline{\mathcal{C}}$ is smooth and has genus $1$.
\end{lemma}

\begin{proof}
On the affine chart $z=1$ we have $\mathcal{C}: v^2=F(y)$ with $F(y)=16y^4+136y^2+1$.
A singular affine point would satisfy $F(y)=0$ and $F'(y)=0$.
But
\[
F'(y)=64y^3+272y=16y(4y^2+17),
\]
and $F(0)=1\neq0$, while $F(y)=-288\neq0$ whenever $4y^2+17=0$.
Hence $\gcd(F,F')=1$ and the affine chart is nonsingular. At $z=0$ we have $v^2=16y^4$, so the two points $P_\infty^\pm$ are also nonsingular.
Therefore $\overline{\mathcal{C}}$ is smooth.

Since $\deg F=4$, the hyperelliptic genus formula gives
$g=\lfloor(\deg F-1)/2\rfloor=1$ for a smooth curve $v^2=F(y)$.
\end{proof}

Let $F(Y)=AY^4+BY^3+CY^2+DY+E$ with $(A,B,C,D,E)=(16,0,136,0,1)$.
The classical invariants of a binary quartic (see, e.g., Cassels~\cite{Cassels} or Silverman~\cite{SilvermanAEC}) are
\[
I=12AE-3BD+C^2=18688.
\]
\[
J=72ACE+9BCD-27AD^2-27EB^2-2C^3=-4\,874\,240.
\]
A Weierstrass model of $\mathrm{Jac}(\mathcal{C})$ is
\begin{equation}\label{eq:E0}
E_0:\ Y^2=X^3-27\,I\,X-27\,J
=X^3-504\,576\,X+131\,604\,480.
\end{equation}
Its $j$-invariant equals $j(E_0)=\frac{1\,556\,068}{81}$.

For arithmetic convenience we work on an equivalent model
\begin{equation}\label{eq:E}
E:\quad Y^2=X(X-8)(X-9)=X^3-17X^2+72X.
\end{equation}
In the Cremona database this curve belongs to isogeny class $48\mathrm{a}$ with label $48\mathrm{a}3$; see~\cite{CremonaEC,LMFDB48a3}.

This curve has the same $j$ and three rational $2$-torsion points $(0,0),(8,0),(9,0)$.
Choosing the rational point $P_\infty^+$ as the origin makes $\overline{\mathcal{C}}$ an elliptic curve and identifies it with its Jacobian over $\mathbb{Q}$ (see \cite{Cassels}).
Proposition~\ref{prop:CEiso} below gives an explicit isomorphism $\overline{\mathcal{C}}\simeq E$ sending $P_\infty^+$ to $O$, and we henceforth view $E$ as a convenient Weierstrass model of $\mathrm{Jac}(\mathcal{C})$ for computations.

\begin{remark}[Explicit isomorphism $E_0\simeq E$]
On $E_0:\ y_0^2=x_0^3-504576\,x_0+131604480$ set
\[
(X,Y)=\bigl((x_0+816)/12^{2},\ y_0/12^{3}\bigr)
\]
(equivalently, $x_0=12^{2}X-816,\ y_0=12^{3}Y$).
A direct substitution yields $Y^{2}=X^{3}-17X^{2}+72X$.
Moreover $\Delta(E_0)=12^{12}\Delta(E)$ and $j(E_0)=j(E)$.
\end{remark}

\begin{proposition}[Explicit isomorphism $\overline{\mathcal{C}}\simeq E$]\label{prop:CEiso}
Define a rational map $\phi:\mathcal{C}\dashrightarrow E$ on the affine chart by
\begin{equation}\label{eq:phi}
\phi(y,v)=(X,Y),\qquad
X=\frac{v+4y^2+17}{2},\qquad
Y=y\,(v+4y^2+17).
\end{equation}
Then $(X,Y)$ satisfies $E:\ Y^2=X(X-8)(X-9)$. Moreover, $\phi$ extends to an isomorphism of smooth projective curves $\overline{\mathcal{C}}\to E$ over $\mathbb{Q}$ sending
\[
P_\infty^+=(1:4:0)\longmapsto O,\qquad
P_\infty^-=(1:-4:0)\longmapsto (0,0).
\]
The inverse map on affine points with $X\neq0$ is
\begin{equation}\label{eq:phi_inv}
y=\frac{Y}{2X},\qquad v=X-\frac{72}{X}.
\end{equation}
\end{proposition}

\begin{proof}
Set $U:=v+4y^2+17$, so that $X=U/2$ and $Y=yU$.
Then
\[
X-8=\frac{v+4y^2+1}{2},\qquad X-9=\frac{v+4y^2-1}{2},
\]
and therefore
\begin{align*}
X(X-8)(X-9)
&=\frac18\,(v+4y^2+17)(v+4y^2+1)(v+4y^2-1)\\
&=\frac18\,U\bigl((v+4y^2)^2-1\bigr).
\end{align*}
Using $v^2=16y^4+136y^2+1$ we compute
\[
(v+4y^2)^2-1=v^2+8y^2v+16y^4-1
=8y^2(v+4y^2+17)=8y^2U,
\]
hence $X(X-8)(X-9)=\frac18\,U\cdot(8y^2U)=y^2U^2=(yU)^2=Y^2$.

From \eqref{eq:phi} we have $U=2X$ and $Y=2Xy$, so $y=Y/(2X)$ for $X\neq0$.
Moreover
\[
v=U-4y^2-17=2X-17-\frac{Y^2}{X^2}.
\]
Using the equation of $E$ we get $Y^2/X^2=X-17+72/X$, hence $v=X-72/X$.
This proves \eqref{eq:phi_inv} and shows that $\phi$ is birational.

Finally, near $P_\infty^+$ one has $v\sim 4y^2$, hence $U\sim 8y^2$ and $X\to\infty$, so $\phi(P_\infty^+)=O$.
Near $P_\infty^-$, the equation of $\mathcal{C}$ gives $v=-4y^2-17+36y^{-2}+O(y^{-4})$. Hence $X=18y^{-2}+O(y^{-4})$ and $Y=36y^{-1}+O(y^{-3})$, so $(X,Y)\to(0,0)$ and $\phi(P_\infty^-)=(0,0)$.
Since both curves are smooth projective curves (lemma~\ref{lem:genus1}), any birational map between them is an isomorphism; see, e.g., \cite[Ch.~I, \S6]{HartshorneAG}. Hence $\phi$ extends to an isomorphism. Since $\phi(P_\infty^+)=O$, it also preserves the chosen origins.
\end{proof}

\begin{remark}[Torsion on $E$]\label{rem:tors}
For $E$ in \eqref{eq:E}, the three points $(0,0),(8,0),(9,0)$ are rational $2$-torsion.
Moreover, the points $(6,\pm 6)$ and $(12,\pm 12)$ have order $4$.
Indeed, for $E:\ Y^2=X^3-17X^2+72X$ the doubling slope is
\[
m=\frac{3X^2-34X+72}{2Y},
\]
and the duplication formulas are
\[
X(2P)=m^2+17-2X,\qquad Y(2P)=-Y+m\bigl(X-X(2P)\bigr).
\]
For $P=(6,6)$ one gets $m=-2$ and hence $2P=(9,0)$; for $P=(12,12)$ one gets $m=4$ and hence $2P=(9,0)$.
Since $(9,0)$ is a nontrivial $2$-torsion point, these points have order $4$.
Thus $E(\mathbb{Q})$ contains a subgroup isomorphic to $\mathbb{Z}/2\mathbb{Z}\oplus\mathbb{Z}/4\mathbb{Z}$; proposition~\ref{prop:rank0} shows this is the full torsion subgroup.
\end{remark}

\section{Rank and torsion of \texorpdfstring{$E$}{E} (via computation)}\label{sec:rank}
\begin{proposition}[Rank $0$ and torsion]\label{prop:rank0}
For the curve $E/\mathbb{Q}$ given by \eqref{eq:E}, one has
\[
\mathrm{rank}\,E(\mathbb{Q})=0,\qquad
E(\mathbb{Q})_{\mathrm{tors}}\cong \mathbb{Z}/2\mathbb{Z}\oplus\mathbb{Z}/4\mathbb{Z},\qquad
\mathrm{Cond}(E)=48.
\]
\end{proposition}

\begin{proof}[Computational proof]
A certified \textsc{Magma} computation (see Appendix~\ref{app:magma})~\cite{Magma} on the model
$E: Y^2 = X(X-8)(X-9)$ returns
\[
\mathrm{Cond}(E)=48,\qquad
E(\mathbb{Q})_{\mathrm{tors}}\cong \mathbb{Z}/2\mathbb{Z}\oplus\mathbb{Z}/4\mathbb{Z},\qquad
\mathrm{rank}\,E(\mathbb{Q})=0.
\]
Moreover, \texttt{MinimalModel(E)} has Cremona label $48\mathrm{a}3$ (Appendix), and the same invariants
(conductor $48$, rank $0$, torsion $\mathbb{Z}/2\oplus\mathbb{Z}/4$) are recorded in Cremona's tables and in the LMFDB entry for $48\mathrm{a}3$; see \cite{CremonaEC,LMFDB48a3}.

The call \texttt{Rank(E)} returns both the value $0$ and the certification flag \texttt{true}; both outputs are recorded in the transcript.
The computed torsion invariants are $[2,4]$, so the torsion subgroup has order $8$.
These computations establish the asserted rank, torsion structure, and conductor.
\end{proof}


\section{Rational points on \texorpdfstring{$\mathcal{C}$}{C}}\label{sec:CQ}
\begin{proposition}\label{prop:Cpoints}
The set $\overline{\mathcal{C}}(\mathbb{Q})$ consists of exactly $8$ points: the two points at infinity $P_\infty^\pm$ and the six affine points
\[
(y,v)=(0,\pm1),\qquad \Bigl(\pm\tfrac12,\ \pm6\Bigr).
\]
\end{proposition}

\begin{proof}
By proposition~\ref{prop:rank0} and remark~\ref{rem:tors}, the Mordell--Weil group is
$E(\mathbb{Q})=E(\mathbb{Q})_{\mathrm{tors}}$ of order $8$. The following eight distinct points belong to this group by Remark~\ref{rem:tors}, so they exhaust it:
\[
E(\mathbb{Q})=\{O,\ (0,0),\ (8,0),\ (9,0),\ (6,\pm6),\ (12,\pm12)\}.
\]
Via the explicit isomorphism $\overline{\mathcal{C}}\simeq E$ from proposition~\ref{prop:CEiso} we obtain:
\[
O\longleftrightarrow P_\infty^+,\qquad (0,0)\longleftrightarrow P_\infty^-,
\]
and for $X\neq0$ the inverse formulas \eqref{eq:phi_inv} give
\[
(8,0)\mapsto (0,-1),\quad (9,0)\mapsto(0,1),\quad
(6,\pm6)\mapsto\bigl(\pm\tfrac12,-6\bigr),\quad
(12,\pm12)\mapsto\bigl(\pm\tfrac12,6\bigr).
\]
Thus $\overline{\mathcal{C}}(\mathbb{Q})$ contains precisely the $8$ points listed. Completeness follows from the computed rank and torsion order and the explicit isomorphism, independently of any bounded point search.
\end{proof}

\begin{corollary}\label{cor:taus}
The only rational values of $\tau=y^2$ arising from $\mathcal{C}(\mathbb{Q})$ are
\[
\tau\in\{0,\ 1/4\}.
\]
\end{corollary}


\section{Excluding the values of \texorpdfstring{$\tau$}{tau}}\label{sec:tau}
Recall $\tau=(au/\Delta)^2$ with $\Delta=u^2-a^2\neq0$ and $au\neq0$.

\begin{lemma}\label{lem:tau}
For rational numbers $a,u$ with $au\neq0$ and $u^2\neq a^2$ one has $\tau\notin\{0,1/4\}$.
\end{lemma}

\begin{proof}
If $\tau=0$, then $au=0$, contradicting the hypotheses.
If $\tau=1/4$, then $(au/\Delta)^2=1/4$, hence $\Delta=\pm 2au$.
Writing $x=u/a\in\mathbb{Q}$, this gives $x^2-1=\pm 2x$, i.e.
\[
x^2-2x-1=0\quad\text{or}\quad x^2+2x-1=0,
\]
whose solutions are $x=1\pm\sqrt2$ or $x=-1\pm\sqrt2$, none of which is rational.
\end{proof}

\begin{corollary}[No $K$-split of $h_\pm$]\label{cor:noKsplit}
Under the standing hypotheses, neither $D_-$ nor $D_+$ in \eqref{eq:DeltaS} is a square in $K=\mathbb{Q}(\sqrt2)$. Hence $h_-(S)$ and $h_+(S)$ do not split in $K[S]$, and each $H_\pm(t)$ is irreducible in $K[t]$.
\end{corollary}

\begin{proof}
Assume for contradiction that $D_-$ is a square in $K$ (equivalently, that $h_-(S)$ splits in $K[S]$). Then, by the algebra of section~\ref{sec:reduction}, there exist $v,\tau\in\mathbb{Q}$ with $v^2=16\tau^2+136\tau+1$ and $\tau=(au/\Delta)^2=y^2$, i.e.\ $\mathcal{C}$ has a rational point with $y^2=\tau$. By corollary~\ref{cor:taus} and lemma~\ref{lem:tau} this is impossible. Applying the conjugation $\sqrt2\mapsto-\sqrt2$ gives the same conclusion for $D_+$ and hence for $h_+(S)$.

Finally, Lemma~\ref{lem:even}, applied to $H_\pm(t)=h_\pm(t^2)$ with $A_0>0$, shows that each $H_\pm$ is irreducible in $K[t]$.
\end{proof}


\section{Irreducibility of \texorpdfstring{$P_{a,u}(t)$}{P(a,u;t)} over \texorpdfstring{$\mathbb{Z}$}{Z}}\label{sec:final}

\begin{lemma}[Quadratic Galois descent]\label{lem:descent}
Let $K/\mathbb{Q}$ be a quadratic extension with nontrivial automorphism $\sigma$.
Let $P(t)\in\mathbb{Q}[t]$ and suppose that in $K[t]$ one has
\[
P(t)=F(t)\,G(t),\qquad G(t)=\sigma(F(t)),
\]
where $F,G\in K[t]$ are irreducible and $\gcd(F,G)=1$ in $K[t]$.
Then $P(t)$ is irreducible in $\mathbb{Q}[t]$.
\end{lemma}

\begin{proof}
Assume $P=AB$ with $A,B\in\mathbb{Q}[t]$ nonconstant.
Viewed in $K[t]$, the polynomial $A$ factors into irreducibles, each of which must be associate to either $F$ or $G$ because $F$ and $G$ are coprime and $FG=P$.
If $F\mid A$ in $K[t]$, then applying $\sigma$ gives $G=\sigma(F)\mid\sigma(A)=A$, hence $FG\mid A$, so $\deg A\ge\deg P$, a contradiction.
Thus neither $F$ nor $G$ divides $A$, and similarly neither divides $B$, impossible since $AB=FG$.
Therefore $P$ is irreducible in $\mathbb{Q}[t]$.
\end{proof}

\begin{proof}[Proof of Theorem~\ref{thm:main}]
Let $K=\mathbb{Q}(\sqrt2)$ and let $\sigma$ be the conjugation $\sqrt2\mapsto-\sqrt2$.
By lemmas~\ref{lem:split} and~\ref{lem:gcd1} we have in $K[t]$ the factorization
\[
P_{a,u}(t)=H_-(t)\,H_+(t),\qquad H_+(t)=\sigma(H_-(t)),\qquad \gcd(H_-,H_+)=1.
\]
By corollary~\ref{cor:noKsplit}, both $H_-$ and $H_+$ are irreducible in $K[t]$.
Therefore lemma~\ref{lem:descent} applies and shows that $P_{a,u}(t)$ is irreducible in $\mathbb{Q}[t]$.
Finally, $P_{a,u}(t)$ is primitive in $\mathbb{Z}[t]$, so Gauss's lemma \cite{LangAlgebra} yields irreducibility in $\mathbb{Z}[t]$.
\end{proof}


\appendix
\section{\textsc{Magma} verification}\label{app:magma}

The script is provided for reproduction with \textsc{Magma} V2.29-10.
The transcript below records the original verification session.
The rank and torsion computations are used in Proposition~\ref{prop:rank0}.
The calls \texttt{IntegralPoints(E)} and \texttt{RationalPoints(C : Bound := 1000)} are auxiliary checks; the latter is a bounded search.
Completeness of the rational-point lists is proved in the main text from the rank, torsion order, and explicit isomorphism, independently of these point-finding calls.

\bigskip
\noindent\textbf{Code.}
\begin{code}
Q := Rationals();
E := EllipticCurve([0,-17,0,72,0]); // Y^2 = X^3 - 17 X^2 + 72 X
Emin, mp := MinimalModel(E); Emin;
CremonaReference(Emin);
Conductor(E);
T := TorsionSubgroup(E); T; AbelianInvariants(T);
Rank(E);
IntegralPoints(E);
P<x> := PolynomialRing(Q);
C := HyperellipticCurve(16*x^4 + 136*x^2 + 1);
EfromC, phi := EllipticCurve(C);
IsIsomorphic(EfromC, E);
RationalPoints(C : Bound := 1000);
\end{code}

\noindent\textbf{Transcript.}
\begin{term}
Elliptic Curve defined by y^2 = x^3 + x^2 - 24*x + 36 over Rational Field
48a3
48
Abelian Group isomorphic to Z/2 + Z/4
Defined on 2 generators
Relations:
    2*T.1 = 0
    4*T.2 = 0
[ 2, 4 ]
0 true
[ (0 : 0 : 1), (6 : -6 : 1), (8 : 0 : 1), (9 : 0 : 1), (12 : 12 : 1) ]
[ <(0 : 0 : 1), 1>, <(6 : -6 : 1), 1>, <(8 : 0 : 1), 1>, <(9 : 0 : 1), 1>, <(12
: 12 : 1), 1> ]
true
{@ (1 : -4 : 0), (1 : 4 : 0), (0 : -1 : 1), (0 : 1 : 1), (-1 : -24 : 2), (-1 :
24 : 2), (1 : -24 : 2), (1 : 24 : 2) @}
\end{term}

\medskip

\noindent
The last set lists the eight rational points on $\overline{\mathcal{C}}$ in \emph{weighted}
projective coordinates $(u:w:z)$ of $\mathbb{P}(1,2,1)$ (weights $1,2,1$). On the affine chart
$z\neq 0$ we have the identification
\[
(y,v)=\bigl(u/z,\ w/z^{2}\bigr).
\]
Thus $(1:\pm 24:2)$ corresponds to $(y,v)=(\tfrac12,\pm 6)$, and the eight rational points are the two points at infinity $(1:\pm 4:0)$ together with the six affine points $(0,\pm 1)$ and $(\pm \tfrac12,\pm 6)$ used in proposition~\ref{prop:Cpoints}.


\end{document}